\documentclass[11pt]{amsart}

\usepackage{amsmath, amssymb, amscd, amsthm}
\usepackage{latexsym}
\usepackage[mathscr]{eucal}
\usepackage{mathrsfs}
\usepackage{graphicx}
\usepackage{verbatim}
\usepackage{version}
\usepackage{nicefrac}
\usepackage{pstricks}
\usepackage{pst-all}
\usepackage{color}
\usepackage[all]{xy}
\usepackage{mathdots}
\usepackage{longtable}
\usepackage{geometry}
\theoremstyle{plain}
\newtheorem{Th}{Theorem}[section]
\newtheorem{Lemma}[Th]{Lemma}
\newtheorem{Cor}[Th]{Corollary}
\newtheorem{Prop}[Th]{Proposition}

 \theoremstyle{definition}

\newtheorem{Rem}[Th]{Remark}
\newtheorem{?}[Th]{Problem}

 \numberwithin{equation}{section}
\begin{document}

\email{Jiseongk@buffalo.edu}
\address{University at Buffalo, Department of Mathematics
244 Mathematics Building
Buffalo, NY 14260-2900}
\title{ON Hecke eigenvalues of cusp forms 
 IN ALMOST ALL SHORT INTERVALS 
}
\author{jiseong kim}
\begin{abstract} 

 Let $\psi$ be a function such that $\psi(x) \rightarrow \infty$ as $x \rightarrow \infty.$ Let $\lambda_{f}(n)$ be the $n$-th Hecke eigenvalue of a fixed holomorphic cusp form $f$ for $SL(2,\mathbb{Z}).$
We show that for any real valued function $h(x)$ such that $(\log X)^{2-2\alpha} \ll h(X) =o(X),$  
$$\sum_{n=x}^{x+h(X)} |\lambda_{f}(n)| \ll_{f} h(X)\psi(X)(\log X)^{\alpha-1}$$  for all but $O_{f}( X\psi(X)^{-2})$ many integers 
  $x\in [X,2X-h(X)],$
in which $\alpha$ is the average value of $|\lambda_{f}(p)|$ over primes. We generalize this for $|\lambda_{f}(n)|^{2^{k}}$ for $k \in \mathbb{Z^{+}}.$
\end{abstract}

\maketitle
\section{Introduction}\

Let $f(z)$ be a holomorphic Hecke cusp form of even integral weight $k
$ for the full modular group $SL(2,\mathbb{Z}).$ Let $e(z)=e^{2\pi iz}$. It is well known that $f(z)$ has a Fourier expansion
\begin{equation}
    f(z)=\sum_{n=1}^{\infty} c_{n}n^{\frac{k-1}{2}}e(nz)
\end{equation}
for some real numbers  $c_{n}.$ For each $n \in \mathbb{N},$
$$T_{n}f(z):=\frac{1}{n} \sum_{ad=n} a^{k} \sum_{0\leq b <d} f(\frac{az+b}{d})=\lambda_{f}(n) f(z),$$ in which $T_{n}$ is the $n$-th Hecke operator, $\lambda_{f}(n)$ is the $n$-th Hecke eigenvalue. The Hecke eigenvalues $\{\lambda_{f}(n)\}_{n \in \mathbb{N}}$ satisfy the following properties.

\begin{equation}
c_{1}\lambda_{f}(n)=c_{n},
\end{equation}
\begin{equation}
    \lambda_{f}(m) \lambda_{f}(n)= \sum_{d|(n,m)} \lambda_{f}(\frac{nm}{d^{2}}),
\end{equation}

\begin{equation}
     |\lambda_{f}(n)| \leq d(n),
\end{equation}
  in which $d(n):= \sum_{m|n} 1$ (the inequality (4) is called the Deligne bound).
For details, see Chapter 14, \cite{IK1}.

We say that $\alpha$ is the average value of $|\lambda_{f}(p)|$ when 
 \begin{equation}
\sum_{p < x} \frac{|\lambda_{f}(p)|}{p} = \sum_{p<x} \frac{\alpha}{p}+O_{f}(1)
\end{equation}
for big enough $X.$
Sato-Tate conjecture implies that $\alpha=\frac{8}{3\pi}$ $(=0.848826...).$ In \cite{EMSS} P. D. T. A Elliott, C. J. Moreno and F. Shahidi proved that $\alpha \leq \frac{17}{18}$ without assuming Sato-Tate conjecture.

When $h=X^{\delta}$  for some  $\delta \in (0,1]$, by Shiu's theorem (see Lemma 2.2),
    $$ \sum_{n=X}^{X+h}|\lambda_{f}(n)| \ll_{f,\delta} h\prod_{p=1}^{X} (1+\frac{\alpha-1}{p})\ll h(\log X)^{\alpha-1}$$ for big enough $X,$
but when $h(X)=o_{\delta}(X^{\delta})$ for any $\delta>0,$ we can not use Shiu's theorem because the interval is too short. In Section 2, we prove some lemmas by using some arguments of the papers \cite{MR3}, \cite{MRT2} to overcome this obstacle.

Although the results in this paper are stated for holomorphic cusp forms, the same arguments in this paper apply equally well to Maass cusp forms on $SL(2,\mathbb{Z}),$ if we assume (1.4)  (The Ramanujan-Petersson conjecture.).

 We give some notations that will be used throughout in this paper. We use $\varphi$ to denote the Euler totient function.
We use $\psi$ to denote a function from $\mathbb{R}$  to $\mathbb{R}$ such that $\psi(x)\rightarrow \infty$ as $x \rightarrow \infty.$
For any two functions $k(x)$ and $l(x)$, we use $k(x)\ll l(x)$ (and $k(x)=O(l(x))$) to denote that there exists a constant $C$ such that  $|k(x)| \leq C l(x)$ for all $x.$ 
We use $k(x)=o(l(x))$ to denote $|\frac{k(x)}{l(x)}| \rightarrow 0$ as $x \rightarrow \infty$ and $n \sim X$ to denote $n \in [X,2X].$ Summing over the index $p$ denotes summing over primes. For the convenience, we denote $h:=h(X).$

\subsection{Main results}

\begin{Th}
Let $X>0$ be big enough,
let $q$ be a natural number smaller than $X.$ Let $h$ be a real valued function such that 
 $\varphi(q)(\log X)^{2-2\alpha} \ll_{f} h = o(X).$ 
  Then there exists a Dirichlet character $\chi$ modulo $q$  such that \begin{equation} \sum_{n=x}^{x+h} |\lambda_{f}(n)|\chi(n) \ll_{f} h\psi(X) \varphi(q)^{-0.5}(\log X)^{\alpha-1} \end{equation}  for all but $O_{f}(X\psi(X)^{-2})$ many integers $x\in [X,2X-h].$
\end{Th}

When $q=1,$ $\chi$ in Theorem 1.1 should be the trivial character. Therefore, we obtain the following corollary. 

\begin{Cor}  Let $X>0$ be big enough. 
 Let $h$ be a real valued function such that 
 $(\log X)^{2-2\alpha} \ll_{f} h = o(X).$ Then   \begin{equation}\sum_{n=x}^{x+h} |\lambda_{f}(n)| \ll_{f} h\psi(X) (\log X)^{\alpha-1}\end{equation}  for all but at most $O_{f}(X\psi(X)^{-2})$ integers  $x\in [X,2X-h].$ 
\end{Cor}

It is well known that for big enough $X,$
\begin{equation}\begin{split} &\sum_{n=1}^{X} |\lambda_{f}(n)|^{2}= c_{1}X+O_{f}(X^{\frac{3}{5}}),\\
&\sum_{n=1}^{X} |\lambda_{f}(n)|^{4}= c_{2}X\log X + c_{3}X+ O_{f,\epsilon}(X^{\frac{7}{8}+\epsilon})
\end{split} \end{equation}
for some $c_{1}, c_{2}, c_{3}$ (see \cite{L1}).
In our method, the upper bound of the short sum (1.7) and the sizes of $h$ in Corollary 1.2 are only depend on the long sums ((1.8), first equation) and the average of $|\lambda_{f}(p)|$ over primes (for the detail, see (2.8)). Therefore, we generalize Corollary 1.2 to arbitrary $2^{k}$ power of $|\lambda_{f}(n)|$ for $k \in \mathbb{Z^{+}}.$ 
\begin{Th} Let $X>0$ be big enough. 
Let  k be a fixed non-negative integer. Assume that there exist positive constants $\beta$ and $\gamma$ such that both inequalities 
\begin{equation} \sum_{n=X}^{2X} |\lambda_{f}(n)|^{2^{k+1}} \ll_{f} X(\log X)^{\beta}, \end{equation}
\begin{equation} \sum_{p=1}^{X} \frac{|\lambda_{f}(p)|^{2^{k}}}{p}-  \sum_{p=1}^{X} \frac{\gamma}{p} = O_{f}(1) \end{equation}
hold. Then for any real valued function $h$ such that 
 $ (\log X)^{\beta-2\gamma+2} \ll_{f} h =o(X),$  $$\sum_{n=x}^{x+h} |\lambda_{f}(n)|^{2^{k}} \ll_{f} h(\log X)^{\gamma-1}\psi(X)$$  for all but $O_{f}(X\psi(X)^{-2})$ many integers $x\in [X,2X-h].$

\end{Th}

In Lemma 2.5, we prove that the average of $\lambda_{f}(p)^{2}$ over primes is 1. Therefore, the upper bound of $\sum_{n=X}^{2X} |\lambda_{f}(n)|^{2}$ from Shiu's theorem is also $O(X).$ From the above facts, we obtain the following corollary.
\begin{Cor}
 Let $X>0$ be big enough. Let $h$ be a real valued function such that 
 $\log X \ll_{f} h = o(X).$ Then  \begin{equation}\sum_{n=x}^{x+h} |\lambda_{f}(n)|^{2} \ll_{f} h\psi(X) \end{equation}  for all but $O_{f}(X\psi(X)^{-2})$ many integers  $x\in [X,2X-h].$
\end{Cor}

\begin{Rem}
We apply Shiu's theorem to get some trivial bounds.
Let $$R_{1}(x):=\sum_{n=x}^{x+h} |\lambda_{f}(n)|\chi(n),$$ $$\displaystyle K_{1}(X):= \{ x \in [X,2X-h] :  h \psi(X) \varphi(q)^{-0.5}(\log X)^{\alpha-1}\ll R_{1}(x)\}.$$ Then 
\begin{equation} 
\begin{split}
|K_{1}(X)| h\psi(X) \varphi(q)^{-0.5}(\log X)^{\alpha-1} & \leq \sum_{X \leq x \leq 2X \atop h \psi(X) \varphi(q)^{-0.5}(\log X)^{\alpha-1}\ll R_{1}(x)} |R_{1}(x)| \\
& \leq \sum_{X \leq x \leq 2X} |R_{1}(x)| \\ 
& \leq \sum_{X \leq x \leq 2X} \sum_{n=x}^{x+h} |\lambda_{f}(n)| \\
& \ll_{f} hX(\log X)^{\alpha-1}.
\end{split}
\end{equation}
Therefore, $X (\psi(X) \varphi(q)^{-0.5})^{-1}$ is a trivial bound for $|K_{1}(X)|. $
Thus the upper bound of $|K_{1}(X)|$ from Corollary 1.2 saves $\psi(X)\varphi(q)^{0.5}$\ from the trivial one.

Let $$ R_{2} (x):=\sum_{n=x,(n,q)=1}^{x+h} \,|\,\lambda_{f}(n)|^{2}, $$   $$ K_{2}(X):= \{ x \in [X,2X-h]: h \psi(X) \ll R_{2}(x) \}.$$
Then 
\begin{equation} 
\begin{split}
|K_{2}(X)| h\psi(X)  & \leq \sum_{X \leq x \leq 2X \atop h \psi(X) \ll R_{2}(x)} |R_{2}(x)| \\
& \leq \sum_{X \leq x \leq 2X} |R_{2}(x)| \\ 
& \leq \sum_{X \leq x \leq 2X} \sum_{n=x}^{x+h} |\lambda_{f}(n)|^{2} \\
& \ll_{f} hX
\end{split}
\end{equation} 
Therefore, $X(\psi(X))^{-1} $ is a trivial bound for $|K_{2}(X)|.$ Thus the upper bound of $|K_{2}(X)|$ from Corollary 1.4 saves $\psi(X)$ from the trivial one.
\end{Rem}



\section{Lemmas}
The following lemma shows that one can get some information about the average of $|\lambda_{f}(n)|\chi(n)$ in almost all short intervals from the upper bounds of the second moment of the Dirichlet polynomial
$$F(s):= \sum_{n \sim X} \frac{|\lambda_{f}(n)|\chi(n)}{n^{s}}.$$

\begin{Lemma}
Let $X>0$ be big enough, let $q$ be a natural number smaller than $X,$ and let $h=o(X).$ Then

\begin{equation}
\frac{1}{X} \int_{X}^{2X} \Big|\frac{1}{h}\sum_{n=x}^{x+h} |\lambda_{f}(n)|\chi(n) \Big|^{2}dx \end{equation}
$$\ll \int_{0}^{Xh^{-1}} \Big|\sum_{n \sim X} \frac{|\lambda_{f}(n)|\chi(n)}{n^{1+it}}\Big|^{2}dt + \max_{T>Xh^{-1}} \frac{Xh^{-1}}{T} \int_{T}^{2T} \Big|\sum_{n\sim X} \frac{|\lambda_{f}(n)|\chi(n)}{n^{1+it}}\Big|^{2}dt. $$
\end{Lemma}

\begin{proof}
The proof of this basically follows from \cite[Lemma 14]{MR3}.
Since we choose the Dirichlet polynomial $F(s)$ instead of $\sum_{n=1}^{\infty} \frac{|\lambda_{f}(n)|\chi(n)|}{n^{1+it}},$ there is no issue on absolute convergence of $F(s).$
By Perron's formula, $$ \sum_{x \leq n \leq x+h} |\lambda_{f}(n)|\chi(n)= \frac{1}{2\pi i} \int_{1-i\infty}^{1+i\infty} F(s) \frac{(x+h)^{s}-x^{s}}{s}ds.$$
Let 
$$V= \frac{1}{h^{2}X} \int_{X}^{2X} \Big|\int_{1}^{1+i\infty} F(s) \frac{(x+h)^{s}-x^{s}}{s} ds \Big|^{2}dx.$$
Since $$\frac{(x+h)^{s}-x^{s}}{s} = \frac{1}{2h} \Big[\int_{h}^{3h} \frac{(x+w)^{s}-x^{s}}{s} dw - \int_{h}^{3h} \frac{(x+w)^{s}-(x+h)^{s}}{s}dw\Big],$$

\begin{equation}
\begin{split}V &\ll Xh^{-4} \int_{X}^{2X} \Big| \int_{\frac{h}{x}}^{\frac{3h}{x}} \int_{1}^{1+i\infty} F(s)x^{s} \frac{(1+w)^{s}-1}{s}dsdw\Big|^{2}dx
\\&+ Xh^{-4} \int_{X}^{2X} \Big| \int_{0}^{\frac{2h}{x+h}} \int_{1}^{1+i\infty} F(s)(x+h)^{s} \frac{(1+w)^{s}-1}{s}dsdw\Big|^{2}dx.
\end{split}
\end{equation}
By the mean value theorem, the right hand side of (2.2) is bounded by 
\begin{equation} \begin{split} \ll& \frac{1}{h^{2}X} \int_{X}^{2X} \Big| \int_{1}^{1+i\infty} F(s) x^{s} \frac{(1+u)^{s}-1}{s} ds\Big|^{2}dx
\\&+ \frac{1}{h^{2}X} \int_{X+h}^{2X+2h} \Big| \int_{1}^{1+i\infty} F(s) x^{s} \frac{(1+u)^{s}-1}{s} ds\Big|^{2}dx
\end{split} \end{equation}
for some $u\ll \frac{h}{X}.$
Let $V_{1}$ be the first summand, $V_{2}$ be the second summand of (2.3).
Let $g_{1}$ be  a smooth function supported on $[\frac{X}{2},4X], $  $g_{1}(x)=1$ for $ x \in [X,2X],$ and $g_{1}'(x) \ll \frac{1}{X}.$ Let $s_{1}=1+it_{1}, s_{2}=1+it_{2}.$  Then 

\begin{equation}
\begin{split}
V_{1} &\ll \frac{1}{h^{2}X} \int g_{1}(x)\Big| \int_{1}^{1+i\infty} F(s)x^{s} \frac{(1+u)^{s}-1}{s} ds|^{2} dx \\
   &\ll \frac{1}{h^{2}X} \int_{1}^{1+i\infty} \int_{1}^{1+i\infty} \Big|F(s_{1})\overline{F(s_{2})} \min \{ \frac{h}{X},\frac{1}{|t_{1}|}\} \min \{ \frac{h}{X},\frac{1}{|t_{2}|}\} \Big|\Big|\int g_{1}(x)x^{s_{1}+\bar{s_{2}}}dx \Big||ds_{1}ds_{2}|.
\end{split} \nonumber
\end{equation}
Since 
$$\int g_{1}(x)x^{s_{1}+\bar{s_{2}}}dx \ll  \frac{1}{X}\int_{\frac{X}{2}}^{4X} \Big|\frac{x^{s_{1}+\bar{s_{2}}+1}}{s_{1}+\bar{s_{2}}+1}\Big|dx,$$

\begin{equation}
\begin{split}
V_{1} &\ll \frac{1}{h^{2}X} \int_{1}^{1+i\infty} \int_{1}^{1+i\infty} \Big|F(s_{1})\overline{F(s_{2})} \min \{ \frac{h}{X},\frac{1}{|t_{1}|}\} \min \{ \frac{h}{X},\frac{1}{|t_{2}|}\} \frac{X^{3}}{\sqrt{|t_{1}-t_{2}|^{2}+1}}\Big| |ds_{1}ds_{2}| \\ 
& \ll \frac{X^{2}}{h^{2}} \int_{1}^{1+i\infty} \int_{1}^{1+i\infty} \frac{|F(s_{1})|^{2} \min \{(\frac{h}{X})^{2}, |t_{1}|^{-2}\} +|F(s_{2})|^{2} \min \{(\frac{h}{X})^{2}, |t_{2}|^{-2}\}}{\sqrt{|t_{1}-t_{2}|^{2}+1}} |ds_{1}ds_{2}|\\
&\ll \int_{1}^{1+i\frac{X}{h}} |F(s)|^{2} |ds| + \frac{X^{2}}{h^{2}} \int_{1+\frac{iX}{h}}^{1+i\infty} \frac{|F(s)|^{2}}{|t|^{2}}|ds|.
\end{split}
\nonumber
\end{equation}
Since $|t|^{-2}\ll \int_{it}^{2it} |T|^{-3} dT,$
\begin{equation}
\begin{split}
V_{1} &\ll  \int_{1}^{1+i\frac{X}{h}} |F(s)|^{2} |ds|  + \frac{X^{2}}{h^{2}} \int_{\frac{X}{2h}}^{\infty} \frac{1}{T^{3}} \int_{1+iT}^{1+2iT} |F(s)|^{2} |ds||dT|\\
&\ll   \int_{1}^{1+i\frac{X}{h}} |F(s)|^{2} |ds|  + \frac{X^{2}}{h^{2}} \frac{h}{X} \max_{T > \frac{X}{2h}} \frac{1}{T} \int_{1+iT}^{1+2iT} |F(s)|^{2} |ds| \\
&\ll \int_{0}^{Xh^{-1}} \Big|\sum_{n=X}^{2X} \frac{|\lambda_{f}(n)|\chi(n)}{n^{1+it}}\Big|^{2}dt + \max_{T>Xh^{-1}} \frac{Xh^{-1}}{T} \int_{T}^{2T} \Big|\sum_{n=X}^{2X} \frac{|\lambda_{f}(n)|\chi(n)}{n^{1+it}}\Big|^{2}dt .
\end{split} \nonumber
\end{equation}
Let $g_{2}$ be  a smooth function supported on $[\frac{X+h}{2},4X+4h], $  $g_{2}(x)=1$ for $ x \in [X+h,2X+2h],$ and $g_{2}'(x) \ll \frac{1}{X}.$ 
By the similar arguments of the bounding $V_{1}$ (replacing $g_{1}$ with $g_{2}$), 

\begin{equation}
    V_{2} \ll \int_{0}^{Xh^{-1}} \Big|\sum_{n=X}^{2X} \frac{|\lambda_{f}(n)|\chi(n)}{n^{1+it}}\Big|^{2}dt + \max_{T>Xh^{-1}} \frac{Xh^{-1}}{T} \int_{T}^{2T} \Big|\sum_{n=X}^{2X} \frac{|\lambda_{f}(n)|\chi(n)}{n^{1+it}}\Big|^{2}dt .
\nonumber\end{equation}
 
\end{proof}

In the proof of Lemma 2.3, we bound some type of the integral \begin{equation} \int_{-T}^{T} |F(1+it)|^{2}dt \end{equation} by some terms in which are related to the average of $|\lambda_{f}(n)|^{2}$ over $[X,2X]$ and the average of the shifted sums $\sum_{X \leq m \leq 2X} |\lambda_{f}(m) \lambda_{f}(m+hq)|$ over $h\in [1,\frac{T}{Xq}].$ The following lemma allows us to compute them.
\begin{Lemma} \rm{(Shiu's theorem} \cite[Lemma 2.3]{MRT2})

Let $0<\delta \leq 1.$ Let $1 \leq  q \leq X^{\delta},$ $1 \leq H.$ Let $r(n)$ be a non-negative multiplicative function such that $r(n) \ll d(n)^{k}$ for some $k \in \mathbb{N}$. 
For $2\leq X^{\delta}\leq Y,$ 
\begin{equation}
\sum_{n=X}^{X+Y} r(n) \ll_{\delta} Y  \prod_{p<X} (1+\frac{r(p)-1}{p}),
\end{equation}
\begin{equation}\sum_{|h|\leq H} \sum_{X \leq n \leq X+Y \atop (n,q)=1} r(n)r(n+hq) \ll_{\delta} HY \prod_{p\leq X, \atop p \nmid q}(1+\frac{r(p)-1}{p})^{2}\prod_{p \mid q} (1-\frac{1}{p}).
\end{equation}
\end{Lemma}
\begin{proof}
See \cite[Lemma 2.3]{MRT2}.
\end{proof}
Let 
\begin{equation} A(s):= \sum_{n \sim X} a_{n}n^{-s} \end{equation} for some $\{a_{n}\} \in \mathbb{C}.$ The standard method for bounding the second moment of $A(s)$ is the mean value theorem
$$ \int_{-T}^{T} |A(it)|^{2}dt = O\big((T+X)\sum_{n \sim X} |a_{n}|^{2}\big)$$ (see \cite[Theorem 9.2]{IK1}). By factoring $A(s)$ to reduce the size of the length of the Dirichlet polynomials, one can obtain some nontrivial bounds of the second moment of $A(s)$ from the above mean value theroem (see Section 5, \cite{MR3}. In \cite{MR3}, K. Matom\"a{}ki, M. Radziwi\l{}\l{} applied the Ramare identity, an analogue of the Buchstab identity). But in our case, we are unable to reduce the size of $X.$ Therefore, we apply the following lemma.

\begin{Lemma}
Let $X>0$ be big enough, let $q$ be a natural number smaller than $X.$ Then
\begin{equation}
\sum_{\chi (\mathrm{mod} \thinspace q)}\int_{-T}^{T} \Big|\sum_{n\sim X} \frac{|\lambda_{f}(n)|\chi(n)}{n^{1+it}}\Big|^{2}dt \ll_{f} \frac{T\varphi(q)}{X^{2}}\sum_{n \sim X} |\lambda_{f}(n)|^{2} + (\log X)^{2\alpha-2}.
\end{equation}
\end{Lemma}
\begin{proof}
The proof of this basically follows from \cite[Lemma 5.2]{MRT2}.
Let $$I:= \sum_{\chi(\mathrm{mod} \thinspace q)}\int_{-T}^{T} \Big|\sum_{n\sim X} \frac{|\lambda_{f}(n)|\chi(n)}{n^{1+it}}\Big|^{2}dt.$$
Let $\phi$ be a non-negative smooth function such that $\phi \geq 1$ for $|x| \leq 1, \hat{\phi}(x)=0$ for $1<|x|,$ in which $$\hat{\phi}(x):= \int_{-\infty}^{\infty} \phi(t)e(-xt)dt.$$ 
Then 
$$I \leq \sum_{\chi (\mathrm{mod} \thinspace q)} \int_{-\infty}^{\infty} \Big| \sum_{ n \sim X} \frac{|\lambda_{f}(n)|\chi(n)}{n^{1+it}}\Big|^{2} \phi(\frac{t}{T}) dt $$
$$=\sum_{\chi (\mathrm{mod} \thinspace q)} \sum_{m,n\sim X} \frac{|\lambda_{f}(m) \lambda_{f}(n)|}{(mn)}\chi(m)\overline{\rm \chi(n)} T \hat{\phi}(T\log (\frac{m}{n})).$$
For each fixed $n$, the range of $m$ is decided by the compact support of $\hat{\phi}$  $(m =n+h, |h|\leq \frac{2X}{T})$, and by averaging over characters $\chi (\mathrm{mod} \thinspace q)$, 

\begin{equation} I \ll \varphi(q)\frac{T}{X^{2}} \sum_{n \sim X \atop (n,q)=1} |\lambda_{f}(n)|^{2} + \varphi(q)\frac{T}{X^{2}}\sum_{0<|h|< \frac{2X}{Tq}} \sum_{ n\sim X \atop (n,q)=1} |\lambda_{f}(n)\lambda_{f}(n+hq)|.\nonumber\end{equation} 
By Lemma 2.2,(2.6),
 $$\sum_{0<|h|< \frac{2X}{Tq}} \sum_{n \sim X \atop (n,q)=1} |\lambda_{f}(n)\lambda_{f}(n+hq)|\ll_{f} \frac{2X}{Tq}X\prod_{p \leq X \atop p\nmid q}(1+\frac{|\lambda_{f}(p)|-1}{p})^{2}\prod_{p \mid q} (1-\frac{1}{p}).$$ 
By (1.4), $||\lambda_{f}(p)|-1|\leq 1$ for all prime $p.$ By Taylor expansion and (1.5),
\begin{equation} \begin{split}
 \log\big(\prod_{p \leq X }(1+\frac{|\lambda_{f}(p)|-1}{p})\big) &=  \sum_{p \leq X} \log (1+\frac{|\lambda_{f}(p)|-1}{p}) \\& =  \sum_{p \leq X} \frac{|\lambda_{f}(p)|-1}{p}+O(1)  
 \\&= \sum_{p \leq X} \frac{\alpha-1}{p}+O_{f}(1).
 \end{split} \nonumber \end{equation}
By (1.4), \begin{equation} \begin{split} \log(\frac{\varphi(q)}{q} \prod_{p|q}(1+\frac{|\lambda_{f}(p)|-1}{p})^{-2}(1-\frac{1}{p}))&=\sum_{p|q} \log (1-\frac{1}{p})^{2}(1+\frac{|\lambda_{f}(p)|-1}{p})^{-2} \\& = \sum_{p|q} \frac{-2|\lambda_{f}(p)|}{p}+O(1) \\& \ll 1. 
 \end{split}  \end{equation}
Therefore, the 2nd term of the right-hand side of (2.8) is bounded by  $$(\log X)^{2\alpha-2}.$$
\end{proof}

 Notice that the absolute constant of the inequality (2.8) does not depend on $q.$ but in (2.9), one can produces a saving factor from $\sum_{p|q} \frac{-2|\lambda_{f}(p)|}{p}$ for some $q.$ This saving factor can be crucial when we treat $|\lambda_{f}(n)|^{2^{k}}$ for some big $k.$ Let $k \in \mathbb{N},$  $|\lambda_{f}(2)|=2,$ $q=2.$ Then 
 
 \begin{equation} \begin{split} \log\big(\frac{\varphi(q)}{q} \prod_{p|q}(1+\frac{|\lambda_{f}(p)|^{2^{k}}-1}{p})^{-2}(1-\frac{1}{p})\big)&=\ \log (1-\frac{1}{2})^{2}(1+\frac{2^{2^{k}}-1}{2})^{-2} \\& = -2\log (2^{2^{k}}+1).
 \end{split}  \end{equation}
 Therefore, the last term of (2.10) is heavily depend on $k.$ So we only generalize Lemma 2.3 for $q=1.$   
 
 \begin{Lemma} Let $X>0$ be big enough. 
Let  k be a fixed non-negative integer. Assume that there exist positive constants $\beta$ and $\gamma$ such that both inequalities 
\begin{equation} \sum_{n=X}^{2X} |\lambda_{f}(n)|^{2^{k+1}} \ll_{f} X(\log X)^{\beta}, \end{equation}
\begin{equation} \sum_{p=1}^{X} \frac{|\lambda_{f}(p)|^{2^{k}}}{p}-\sum_{p=1}^{X} \frac{\gamma}{p}= O_{f}(1) \end{equation} hold.
Then
\begin{equation}
\int_{-T}^{T} \Big|\sum_{n\sim X} \frac{|\lambda_{f}(n)|^{2^{k}}}{n^{1+it}}\Big|^{2}dt \ll_{f} \frac{T}{X} (\log X)^{\beta} + (\log X)^{2\gamma-2}.
\end{equation}
 \end{Lemma}
 \begin{proof}
 Let $$I_{k}:= \int_{-T}^{T} \Big|\sum_{n\sim X} \frac{|\lambda_{f}(n)|^{2^{k}}}{n^{1+it}}\Big|^{2}dt.$$
 By the similar argument for $I$ in Lemma 2.3 ($q=1$),
 
  \begin{equation} I_{k} \ll \frac{T}{X^{2}} \sum_{n \sim X} |\lambda_{f}(n)|^{2^{k+1}} + \frac{T}{X^{2}}\sum_{0<|h|< \frac{2X}{T}} \sum_{ n\sim X} |\lambda_{f}(n)^{2^{k}}\lambda_{f}(n+h)^{2^{k}}|.\end{equation} 
 By (2.11),  
 $$\frac{T}{X^{2}} \sum_{n \sim X} |\lambda_{f}(n)|^{2^{k+1}} \ll_{f} \frac{T}{X}(\log X)^{\beta}.$$ 
 By Lemma 2.2 (2.6), 
 \begin{equation} \begin{split}\sum_{0<|h|< \frac{2X}{T}} \sum_{n \sim X} |\lambda_{f}(n)^{2^{k}}\lambda_{f}(n+h)^{2^{k}}|&\ll \frac{2X}{T}X\prod_{p \leq X}(1+\frac{|\lambda_{f}(p)|^{2^{k}}-1}{p})^{2}\\& \ll \frac{X^{2}}{T} (\log X)^{2\gamma-2}. \end{split}\end{equation}
 Therefore, 
 $$I_{k} \ll_{f} \frac{T}{X} (\log X)^{\beta} + (\log X)^{2\gamma-2}.$$

 \end{proof}
The following lemma shows that the average of $\lambda_{f}(p)^{2}$ over primes is 1. 
\begin{Lemma} Let $X>0$ be big enough. Then
$$\sum_{p<X} \frac{\lambda_{f}(p)^{2}}{p}= \sum_{p<X}\frac{1}{p} +O_{f}(1)$$
\end{Lemma}
\begin{proof}
Let $$L(g,s):= \sum_{n=1}^{\infty} \lambda_{f}(n)n^{-s}.$$
Let $\Lambda$ be the von Mangoldt function. 
$L(g \otimes \bar{g},s)$ has a zero free region by \cite[Theorem 5.44]{IK1}.
By \cite[Theorem 5.13]{IK1}, 
$$\sum_{p \leq x} \lambda_{f}(p)^{2}\Lambda(p) = x + O(x(\log x) e^{-C\log ^{\frac{1}{2}}x})$$
for some absolute constant $C>0$ depending only on $g.$ Partial summation over $p$ gives  
$$\sum_{1< p \leq x} \frac{\lambda_{f}(p)^{2}}{p} = \int_{2}^{x} \frac{1}{t\log t} d(\sum_{p \leq t} \lambda_{f}(p)^{2}\Lambda(p)) +O_{f}(1)= \log \log x +O_{f}(1)=\sum_{p \leq x} \frac{1}{p} +O_{f}(1).$$ 

\end{proof}

\section{Propositions}
In this section, we prove Proposition 3.1, Proposition 3.2. We need Proposition 3.1, Proposition 3.2 for Theorem 1.1, Theorem 1.3 respectively.  
\begin{Prop}Let $X>0$ be big enough, let $q$ be a natural number smaller than $X.$  Then there exists a Dirichlet character  $\chi$ modulo $q$ such that
 when $\varphi(q)(\log X)^{2-2\alpha}\ll_{f} h=o(X),$
$$\frac{1}{X} \int_{X}^{2X} \big|\frac{1}{h}\sum_{n=x}^{x+h} |\lambda_{f}(n)|\chi(n)\big|^{2} dx \ll_{f} \varphi(q)^{-1} (\log X)^{2\alpha-2}.$$

\end{Prop}
\begin{proof}
Dropping all but one term, there exists a character $\chi$ modulo $q$ such that for all $T>0,$  
\begin{equation} 
\int_{-T}^{T} \Big|\sum_{n\sim X} \frac{|\lambda_{f}(n)|\chi(n)}{n^{1+it}}\Big|^{2}dt \leq \frac{1}{\varphi(q)}\sum_{\chi(\text{mod} \thinspace q)}\int_{-T}^{T} \Big|\sum_{n\sim X} \frac{|\lambda_{f}(n)|\chi(n)}{n^{1+it}}\Big|^{2}dt.
\end{equation}

By Lemma 2.3, (2.8),
\begin{equation}\begin{split}\frac{1}{\varphi(q)}\sum_{\chi(\text{mod} \thinspace q)}\int_{0}^{Xh^{-1}} \big|\sum_{n\sim X} \frac{|\lambda_{f}(n)|\chi(n)}{n^{1+it}}\big|^{2} dt &\ll_{f} \frac{1}{Xh}\sum_{n \sim X} |\lambda_{f}(n)|^{2} + \varphi(q)^{-1}(\log X)^{2\alpha-2}\\&\ll  \frac{1}{h}+ \varphi(q)^{-1}(\log X)^{2\alpha-2} \\&\ll \varphi(q)^{-1} (\log X)^{2\alpha-2}. \end{split} \end{equation}
By the similar argument of (3.2),
\begin{equation}
\begin{split}
\frac{1}{\varphi(q)}\sum_{\chi(\text{mod} \thinspace q)}\max_{T>Xh^{-1}} &\frac{Xh^{-1}}{T} \int_{T}^{2T} \big|\sum_{n\sim X} \frac{|\lambda_{f}(n)|\chi(n)}{n^{1+it}}\big|^{2} dt 
\\&\ll_{f} \max_{T>Xh^{-1}} Xh^{-1}T^{-1}(\frac{T}{X} + \varphi(q)^{-1}(\log X)^{2\alpha-2}) \\ &\ll  \frac{1}{h}+ \varphi(q)^{-1}(\log X)^{2\alpha-2} \\&\ll \varphi(q)^{-1} (\log X)^{2\alpha-2}.
\end{split} \nonumber \end{equation}
By Lemma 2.1, 
\begin{equation} \begin{split}\frac{1}{X} \int_{X}^{2X} |\frac{1}{h}\sum_{n=x}^{x+h} \big|\lambda_{f}(n)|\chi(n) \big|^{2}dx
\ll& \int_{0}^{Xh^{-1}} \big|\sum_{n\sim X} \frac{|\lambda_{f}(n)|\chi(n)}{n^{1+it}}\big|^{2}dt \\&  + \max_{T>Xh^{-1}} \frac{Xh^{-1}}{T} \int_{T}^{2T} \big|\sum_{n \sim X} \frac{|\lambda_{f}(n)|\chi(n)}{n^{1+it}}\big|^{2}dt. 
\end{split}\end{equation}
Therefore, $$\frac{1}{X} \int_{X}^{2X} \big|\frac{1}{h}\sum_{n=x}^{x+h} |\lambda_{f}(n)|\chi(n)\big|^{2} dx \ll_{f} \varphi(q)^{-1} (\log X)^{2\alpha-2}.$$
\end{proof}
By the similar arguments of the proof of Proposition 3.1, we generalize Proposition 3.1 to arbitrary $2^{k}$ power of $|\lambda_{f}(n)|.$
\begin{Prop}
 Let $X>0$ be big enough. 
Let  k be a fixed non-negative integer. Assume that there exist positive constants $\beta$ and $\gamma$ such that both inequalities 
\begin{equation} \sum_{n=X}^{2X} |\lambda_{f}(n)|^{2^{k+1}} \ll_{f} X(\log X)^{\beta}, \end{equation}
\begin{equation} \sum_{p=1}^{X} \frac{|\lambda_{f}(p)|^{2^{k}}}{p}-\sum_{p=1}^{X} \frac{\gamma}{p} = O_{f}(1) \end{equation} hold.
Then for any real valued function $h$ such that 
 $(\log X)^{\beta-2\gamma+2} \ll_{f} h=o(X),$ 

\begin{equation} \frac{1}{X} \int_{X}^{2X} \big|\frac{1}{h}\sum_{n=x}^{x+h} |\lambda_{f}(n)|^{2^{k}}\big|^{2} dx \ll_{f} (\log X)^{2\gamma-2}.\end{equation}

\end{Prop}
\begin{proof}
By the similar argument of the proof of Lemma 2.1 (one just need to replace $|\lambda_{f}(n)|$ with $|\lambda_{f}(n)|^{2^{k}}$, $q=1$), 
\begin{equation}
\frac{1}{X} \int_{X}^{2X} \big|\frac{1}{h}\sum_{n=x}^{x+h} |\lambda_{f}(n)|^{2^{k}} \big|^{2}dx \end{equation}
$$\ll \int_{0}^{Xh^{-1}} \big|\sum_{n \sim X} \frac{|\lambda_{f}(n)|^{2^{k}}}{n^{1+it}}\big|^{2}dt + \max_{T>Xh^{-1}} \frac{Xh^{-1}}{T} \int_{T}^{2T} \big|\sum_{n\sim X} \frac{|\lambda_{f}(n)|^{2^{k}}}{n^{1+it}}\big|^{2}dt. $$
By Lemma 2.4,

\begin{equation}\begin{split}\int_{0}^{Xh^{-1}} \big|\sum_{n\sim X} \frac{|\lambda_{f}(n)|^{2^{k}}}{n^{1+it}}\big|^{2} dt &\ll_{f}  h^{-1}(\log X)^{\beta}+ (\log X)^{2 \gamma -2}, \end{split} \end{equation}
\begin{equation}
\begin{split}
\max_{T>Xh^{-1}} \frac{Xh^{-1}}{T} \int_{T}^{2T} \big|\sum_{n\sim X} \frac{|\lambda_{f}(n)|^{2^{k}}}{n^{1+it}}\big|^{2} dt 
&\ll_{f} \max_{T>Xh^{-1}} Xh^{-1}T^{-1}(\frac{T}{X}(\log X)^{\beta} +(\log X)^{2\gamma-2}) \\ &\ll  h^{-1} (\log X)^{\beta}+ (\log X)^{2\gamma-2}.
\end{split}  \end{equation}
Since $h^{-1} \ll_{f} (\log X)^{-\beta+2\gamma-2},$ (3.8),(3.9) are bounded by
$$(\log X)^{2\gamma-2}.$$

\end{proof}

\section{Proof of Theorem 1.1, Theorem 1.3, Corollary 1.4}
\subsection{Proof of Theorem 1.1}
By Proposition 3.1, there exists a $\chi$ modulo $q$ such that
\begin{equation}
\int_{X}^{2X} \big|\frac{1}{h}\sum_{n=x}^{x+h} |\lambda_{f}(n)|\chi(n)\big|^{2}dx \ll_{f} \varphi(q)^{-1}(\log X)^{2\alpha-2}. \nonumber
\end{equation}
Let $B(X)=\psi(X)^{2}\varphi(q)^{-1} (\log X)^{2\alpha-2}.$
By the Chebyshev inequality,
\begin{equation} \begin{split}\big|\{ x\in [X,2X-h] :  |\frac{1}{h} &\sum_{n=x}^{x+h} |\lambda_{f}(n)|\chi(n)| \gg_{f}  B(X)^{\frac{1}{2}}\}\big| \\&\ll_{f} B(X)^{-1}\int_{X}^{2X} \big|\frac{1}{h}\sum_{n=x}^{x+h} |\lambda_{f}(n)|\chi(n)\big|^{2}dx \\&=O_{f}(X\psi(X)^{-2}). \end{split} \nonumber\end{equation}

\subsection{Proof of Theorem 1.3, Corollary 1.4}
By Proposition 3.2, 
\begin{equation} \frac{1}{X} \int_{X}^{2X} \big|\frac{1}{h}\sum_{n=x}^{x+h} |\lambda_{f}(n)|^{2^{k}}\big|^{2} dx \ll_{f} (\log X)^{2\gamma-2}. \nonumber \end{equation}
Let $$B_{k}(X)= \psi(X)^{2}(\log X)^{2\gamma-2}.$$
By the Chebyshev inequality,
\begin{equation}\begin{split}\big|\{ x\in [X,2X-h] :  \frac{1}{h} \sum_{n=x}^{x+h} |\lambda_{f}(n)|^{2^{k}} \gg_{f} B_{k}(X)^{\frac{1}{2}}\}\big| &\ll_{f}  B_{k}(X)^{-1}\int_{X}^{2X} \big|\frac{1}{h}\sum_{n=x}^{x+h} |\lambda_{f}(n)|^{2^{k}}\big|^{2}dx \\&=O_{f}(X\psi(X)^{-2}).\nonumber \end{split} \end{equation}
When $k=1,$ by (1.8), $\beta=1.$ And by Lemma 2.5, $\gamma=1.$ Therefore, 
\begin{equation} \begin{split}\big|\{ x\in [X,2X-h] :  \frac{1}{h} \sum_{n=x}^{x+h} |\lambda_{f}(n)|^{2}| \gg_{f}  \psi(X)\big| &\ll_{f} 
(\psi(X))^{-2}\int_{X}^{2X} \big|\frac{1}{h}\sum_{n=x}^{x+h} |\lambda_{f}(n)^{2}\big|^{2}dx \\&= O_{f}(X\psi(X)^{-2}).\nonumber \end{split}\end{equation}

\section{Acknowledgements}
The author would like to thank his advisor Xiaoqing Li, for helpful advice.   The author also thanks the referee for careful reading and pointing out many mistakes.

\bibliographystyle{plain}   
\bibliography{ss}  
\end{document}